\documentclass[11pt]{amsart}

\usepackage{amscd,amssymb,amsopn,amsmath,amsthm,mathrsfs,graphics,amsfonts,enumerate,verbatim,calc
}
\usepackage{url}

\usepackage[OT2,OT1]{fontenc}
\newcommand\cyr{%
\renewcommand\rmdefault{wncyr}%
\renewcommand\sfdefault{wncyss}%
\renewcommand\encodingdefault{OT2}%
\normalfont
\selectfont}
\DeclareTextFontCommand{\textcyr}{\cyr}

\usepackage{amssymb,amsmath}

\DeclareFontFamily{OT1}{rsfs}{}
\DeclareFontShape{OT1}{rsfs}{n}{it}{<-> rsfs10}{}
\DeclareMathAlphabet{\mathscr}{OT1}{rsfs}{n}{it}

\numberwithin{equation}{section}
\newtheorem{theorem}{Theorem}[section]
\newtheorem{lemma}[theorem]{Lemma}
\newtheorem{proposition}[theorem]{Proposition}

\newtheorem{question}{Question}

\theoremstyle{definition}
\newtheorem{definition}[theorem]{Definition}
\newtheorem{remark}[theorem]{Remark}
\theoremstyle{remark}

\newtheorem{acknowledgement}{Acknowledgement}

\newcommand{\im}{\operatorname{Im}}
\renewcommand{\ker}{\operatorname{Ker}}

\newcommand{\Div}{\operatorname{Div}}

\newcommand{\id}{\operatorname{id}}
\newcommand{\Ext}{\operatorname{Ext}}

\newcommand{\Tor}{\operatorname{Tor}}
\newcommand{\Hom}{\operatorname{Hom}}

\newcommand{\coker}{\operatorname{Coker}}

\newcommand{\pr}{\operatorname{pr}}

\newcommand{\Fil}{\operatorname{Fil}}

\newcommand{\MSpec}{\operatorname{M-Spec}}

\newcommand{\vpl}{\operatornamewithlimits{\varprojlim}}

\newcommand{\vil}{\operatornamewithlimits{\varinjlim}}

\newcommand{\fm}{\frak{m}}

\begin{document}
\title[Another proof of the almost purity theorem for perfectoid valuation rings]
{Another proof of the almost purity theorem for perfectoid valuation rings}

\author[S.Ishiro]{Shinnosuke Ishiro}
\address{Department of Mathematics, College of Humanities and Sciences, Nihon University, Setagaya-ku, Tokyo 156-8550, Japan}
\email{shinnosukeishiro@gmail.com}

\author[K.Shimomoto]{Kazuma Shimomoto}
\address{Department of Mathematics, College of Humanities and Sciences, Nihon University, Setagaya-ku, Tokyo 156-8550, Japan}
\email{shimomotokazuma@gmail.com}



\begin{abstract}
The almost purity theorem is central to the geometry of perfectoid spaces and has numerous applications in algebra and geometry. This result is known to have several different proofs in the case that the base ring is a perfectoid valuation ring. We give a new proof by exploiting the behavior of Faltings' normalized length under the Frobenius map. 
\end{abstract}

\maketitle


\section{Introduction}

The almost purity theorem is a key to many foundational results for perfectoid spaces and it has been used to solve many outstanding problems. Recall that the almost purity theorem for perfectoid valuation rings was an intermediate step toward Scholze's proof of the almost purity theorem for general perfectoid rings with the help of algebraization theorems in the style of Elkik and Gabber-Ramero. Now let us turn our attention to the prehistory of the appearance of Scholze's perfectoid spaces. The \textit{normalized length} was introduced in \cite{Fa02} and played a role in the proof of Faltings' approach to the almost purity theorem. However, to the best of authors' knowledge, the normalized length was used only to calculate the local cohomology of certain big rings with Frobenius action (see also \cite{Shi07}). Our aim is to give an application of the normalized length to offer a new proof of the almost purity theorem for perfectoid valuation rings.

\begin{theorem}
Let $V$ be a perfectoid valuation ring (see Definition \ref{perfdval}) with field of fractions $K$, and let $W$ be the integral closure of $V$ in a finite separable extension of fields $K \to L$. 
\begin{enumerate}
\item
The Frobenius endomorphism on $W/pW$ is surjective.

\item
$V \to W$ is faithfully flat and almost finite \'etale.
\end{enumerate}
\end{theorem}

An outline of two different proofs is given in \cite[Theorem 3.7]{Sch12}, where the first proof uses ramification theory \cite[Proposition 6.6.2 and Theorem 6.6.12]{GR03}, while the second proof makes a reduction to the case of absolutely integrally closed perfectoid valuation rings. Our new proof relies on fundamental properties of the normalized length (see \cite{Sch13} and \cite{Shi07}). The reason for an inclusion of this proof is that the authors think that the almost purity theorem has fundamental importance and it is still worth recollecting classical approaches. The structure of this paper goes as follows.

In \S \ref{Valuation}, we give a brief review of valuation rings and some properties of the valuation topology.

In \S \ref{AlmostRing}, we begin with almost ring theory. Some results on almost projective modules that are not explicit in \cite{GR03} are proved, which are of independent interest. Then we introduce the normalized length over a valuation ring of rank one. The most important aspect is contained in the \textit{Frobenius pull-back formula}; see Theorem \ref{Frobeniuspullback}. 

In \S \ref{perfectoidvaluation}, we introduce perfectoid valuation rings with their \textit{tilts}. The details of the proof of the main result are given.

In \S \ref{Valuationflat}, we prove a result on the flatness of the Frobenius map on valuation rings in mixed characteristic. This result arises from the analysis of the main theorem given in \S \ref{perfectoidvaluation}. An open question is also formulated.

\section{Valuation rings}
\label{Valuation}
Let $K$ be a field and $V$ be its subring. Then $V$ is called a \textit{valuation ring} if $x \in V$ or $x^{-1} \in V$ for every nonzero element $x \in K$. As another definition of a valuation ring, one can use a valuation of $K$.

\begin{proposition}
\label{integral1}
Suppose that $V$ is a henselian local domain with field of fractions $K$, and that $L/K$ is an algebraic extension and let $W$ be the integral closure of $V$ in $L$. Then $W$ is a local domain.
\end{proposition}

\begin{proof}
Since $W$ is the integral closure of $V$ in $L$, there exists an inductive system $\{W_\lambda\}_{\lambda \in \Lambda}$ of finite $V$-algebras such that $W=\vil W_\lambda$. Since $V$ is a henselian local ring, every $W_\lambda$ is a direct sum of local rings by \cite[Theorem 10.1.4]{Fo17}. Moreover, since $W_\lambda$ is a domain, it is local. Since a direct limit of local rings is local, $W$ is local. And since $W$ is a subring of $L$, $W$ is also domain.
\end{proof}

\begin{proposition}
\label{valuationext}
Let $K$ be a field, let $v$ be a valuation on $K$, let $V$ be a valuation ring of $v$ and let $L/K$ be an algebraic extension. Let $W$ be the integral closure of $V$ in $L$. If $w$ is a valuation of $L$ that dominates $v$, then the valuation ring of $L$ associated to $w$ is isomorphic to the localization $W_\fm$ for some maximal ideal $\fm \subset W$. In particular, $W$ is itself a valuation ring under the same hypothesis on $V$ as in Proposition \ref{integral1}.
\end{proposition}

\begin{proof}
See \cite[Chapter VI, \S8, Proposition 6]{Bo72} for the proof.
\end{proof}

We recall a valuation topology on a field (\cite[Chapter VI, \S5]{Bo72}). Let $K$ be a field and let $v:K \to \Gamma_K \cup \{0\}$ be its valuation, where $\Gamma_K$ is the value group of $v$. Set $V_\alpha := \{ x \in K~|~v(x) > \alpha \}$ for any $\alpha \in \Gamma_K$. Then the \textit{valuation topology} on $K$ associated to $v$ is the topology on $K$ for which the $V_\alpha$ form a fundamental system of neighborhoods of $0$.

The valuation ring $V$ is called \textit{microbial} if it has a height-one prime ideal. Microbial valuation rings are characterized as follows.

\begin{proposition}
\label{microbialVal}
Let $K$ be a field, let $v$ be a valuation and let $V$ be its valuation ring. Then the following are equivalent.
\begin{enumerate}
\item
$V$ is a microbial valuation ring.

\item
There exist a nonzero element $x \in K$ such that a limit of the sequence of $(x^{n})_{n \geq 0}$ is $0$. Such an element is called a ``topologically nilpotent element''.
\end{enumerate}
\end{proposition}

\begin{proof}
This is well-known and the proof is found in \cite[Theorem I.1.5.4]{Mor20}.
\end{proof}

\begin{proposition}
\label{valtop}
Let $V$ be a microbial valuation ring. Then if $\varpi \in K^{\times}$ is topologically nilpotent, we have $K=V[\frac{1}{\varpi}]$ and the subspace topology on $V$ coincides with the $\varpi$-adic topology. 
\end{proposition}

\begin{proof}
See \cite[Theorem I.1.5.4]{Mor20}.
\end{proof}

\begin{remark}
\label{complete}
By Proposition \ref{valtop}, $V$ is complete with respect to the valuation topology defined by $v$ if and only if $V$ is $\varpi$-adically complete.
\end{remark}

We often use the following notation: Let $A$ be a ring and let $M$ be an $A$-module. Fix a prime $p>0$ and assume that $pM=0$. Then $M^{[F]}$ is defined by setting $M^{[F]}=M$ as additive groups and the $A$-module structure on $M$ is given through the restriction of scalars via the Frobenius endomorphism on $A/pA$.

\section{Almost ring theory}
\label{AlmostRing}
\subsection{Basic notions}
Our references of almost ring theory are \cite{GR03} and \cite{GR18}. We call $(V,I)$ a \textit{basic setup} if $V$ is a commutative ring and $I$ is an ideal such that $I^2=I$ and $I \otimes_{V} I$ is flat. Almost ring theory works under a basic setup $(V,I)$.

\begin{definition}
Let $(V,I)$ be a basic setup, let $R$ and $S$ be $V$-algebras, and let $M$ be a $V$-module.
\begin{enumerate}
\item
$M$ is called \textit{$I$-almost zero} if $IM=0$. We denote this by $M \approx 0$

\item
$M$ is called \textit{$I$-almost finitely generated} if for every finitely generated ideal $I_0 \subseteq I$, there exists a finitely generated $V$-module $M_0$ and a $V$-module maps $f_0 : M_0 \to M$ such that both kernel and cokernel of $f_{0}$ are annihilated by $I_0$. 

\item
$M$ is called \textit{$I$-almost finitely presented} if for every finitely generated ideal $I_0 \subseteq I$, there exists a finitely presented $V$-module $M_{0}$ and a $V$-module maps $f_{0} : M_{0} \to M$ such that both kernel and cokernel of $f_{0}$ are annihilated by $I_0$. 

\item
$M$ is called \textit{$I$-almost flat} if $\Tor^{V}_{i}(M,N) \approx 0$ for any $V$-module $N$ and $i > 0$.

\item
$M$ is called \textit{$I$-almost projective} if $\Ext^{i}_{V}(M,N) \approx 0$ for any $V$-module $N$ and $i > 0$. 

\item
$R \to S$ is called \textit{$I$-almost unramified} if $S$ is an $I$-almost projective $S \displaystyle\otimes_{R} S$-module.

\item
$R \to S$ is called \textit{$I$-almost \'etale} if it is $I$-almost unramified and $S$ is an $I$-almost flat $R$-module.

\item
$R \to S$ is called \textit{$I$-almost finite \'etale} if it is \textit{$I$-almost \'etale} and $S$ is an $I$-almost finitely presented $R$-module.
\end{enumerate}
\end{definition}

\begin{lemma}
\label{almostfg}
Let $(V,I)$ be a basic setup and let $M$ be a $V$-module. Then $M$ is an $I$-almost finitely generated module if and only if for every finitely generated ideal $I_{0} \subseteq I$, there exists a finitely generated submodule $M_{0} \subseteq M$ such that $I_{0}M \subseteq M_{0}$.
\end{lemma}

\begin{proof}
This is \cite[Proposition 2.3.10]{GR03}.
\end{proof}

The following proposition is an almost analogue of Nakayama's lemma.

\begin{proposition}
\label{almostNAK}
Let $V$ be a commutative ring and let $I=\bigcup_{n>0} (\varpi^{\frac{1}{p^{n}}})$ be an ideal of $V$ for some regular element $\varpi \in V$ with a system of $p$-th power roots $\{\varpi^{\frac{1}{p^{n}}} \}$ (so that $(V,I)$ is a basic setup). Let $M$ be a $V$-module which is $\varpi$-adically separated. If $M/\varpi M$ is an $I$-almost finitely generated $V/\varpi V$-module, then $M$ is an $I$-almost finitely generated $V$-module.
\end{proposition}

\begin{proof}
Since $M/\varpi M$ is $I$-almost finitely generated, there exists a finitely generated submodule $M' \subseteq M/\varpi M$ such that ${\varpi}^{\frac{1}{p^n}}(M/\varpi M) \subseteq M'$ for any fixed $n>0$ by Lemma \ref{almostfg}. After writing $M'=(V/\varpi V) {\overline{m}_{n_1}}+\cdots+(V/\varpi V){\overline{m}_{n_r}}$, we have
$$
\varpi^{\frac{1}{p^{n}}}M \subseteq V{m_{n_1}}+\cdots +V{m_{n_r}}+\varpi M \subseteq V{m_{n_1}}+\cdots +V{m_{n_r}}+\bigcap_{k>0} {\varpi}^{k(\frac{p^n -1}{p^{n}})}M,
$$
where $m_{n_{i}}$ is a lift of the element $\overline{m}_{n_{i}}$ and the second inclusion follows from the first inclusion repeatedly. Since $M$ is $\varpi$-adically separated, we have $\bigcap_{k>0} {\varpi}^{k(\frac{p^n -1}{p^{n}})}M=0$. So $\varpi^{\frac{1}{p^{n}}}M \subseteq V{m_{n_1}}+\cdots +V{m_{n_r}}$, as desired.
\end{proof}

\begin{proposition}
\label{almostsplit}
Let $(V,I)$ be a basic setup and let $M$ be an $I$-almost finitely generated $V$-module. Then $M$ is almost projective if and only if for any $\epsilon \in I$, there exist $n(\epsilon) \in \mathbb{N}$ and $V$-linear maps
$$
M \xrightarrow{u_{\epsilon}} V^{n(\epsilon)} \xrightarrow{v_{\epsilon}} M
$$ 
such that $v_{\epsilon} \circ u_{\epsilon} = \epsilon \id_{M}$.
\end{proposition}

\begin{proof}
We refer the reader to \cite[Lemma 2.4.15]{GR03} for the proof.
\end{proof}

Let $V \to W$ be a ring extension. An element $x \in W$ is called \textit{almost integral} over $V$ if there exists a finitely generated $V$-submodule $M \subseteq W$ such that $\sum^{\infty}_{n=0} V x^{n} \subseteq M$. We say that $W$ is \textit{almost integral} over $V$ if every element of $W$ is almost integral over $V$. Now the following lemma holds. We say that $V$ is \textit{completely integrally closed} in $W$ if every element of $W$ that is almost integral over $V$ belongs to $V$.

\begin{lemma}
\label{almostintegral1}
Let $(V,I)$ be a basic setup with  $I=\bigcup _{n>0} (\varpi ^{\frac{1}{p^{n}}})$ for a regular element $\varpi \in V$. Suppose that $W$ is a $V$-algebra such that $W$ is an $I$-almost finitely generated $V$-module and $W$ is $\varpi $-torsion
free. Let
$(V_{W})^{*}$ be the integral closure of $V$ in $W$. Then $W$ is almost
integral over $(V_{W})^{*}$.
\end{lemma}

\begin{proof}
By taking the image of $V$ in $W$, we may assume that $V$ is a subring of $W$. Let $(V_W)^*$ be the integral closure of $V$ in $W$. Since $W$ is an $I$-almost finitely generated $V$-module, there is a finitely generated $V$-module $M_\varpi$ and a $V$-homomorphism $f_\varpi:M_\varpi \to W$ such that $\ker(f_{\varpi})$ and $\coker(f_{\varpi})$ are annihilated by $\varpi$. In particular, we have $M_\varpi[\frac{1}{\varpi}] \cong W[\frac{1}{\varpi}]$ and hence $W[\frac{1}{\varpi}]$ is module-finite over $V[\frac{1}{\varpi}]$. Since $(V_W)^*$ is integrally closed in $W$, it follows that $(V_W)^*[\frac{1}{\varpi}]$ is integrally closed in $W[\frac{1}{\varpi}]$. So we get $(V_W)^*[\frac{1}{\varpi}]=W[\frac{1}{\varpi}]$. Note that $W$ is still an $I$-almost finitely generated $(V_W)^*$-module.

Now we prove that $W$ is almost integral over $(V_W)^*$. To this aim, pick a nonzero $\epsilon \in I$. Then there exist a finitely generated $(V_W)^*$-module $N_{\epsilon}$ and a $(V_W)^*$-homomorphism $f_{\epsilon}:N_{\epsilon} \to W$ such that $\ker(f_{\epsilon})$ and $\coker (f_{\epsilon})$ are annihilated by $\epsilon$. This implies that for every element $x \in W$, $\epsilon x \in \im(f_{\epsilon})$ and thus, $(\epsilon x)^ n \in \im(f_{\epsilon})$ for all $n >0$. Now $\sum^{\infty}_{n=0} (V_W)^* (\epsilon x)^{n} \subset \im(f_{\epsilon}) \subset W$, which implies that $\epsilon x$ is almost integral over $(V_W)^*$. By applying \cite[Lemma 2.3]{NS20}, we find that $\epsilon^2 x=\epsilon(\epsilon x) \in (V_W)^*$. Let us choose $\epsilon=\varpi^{\frac{1}{p^k}}$. Then $\varpi^{\frac{2}{p^k}} x \in (V_W)^*$ and $\varpi^2 x^{p^k} \in (V_W)^*$. Since $k \ge 0$ is arbitrary and $x \in (V_W)^*[\frac{1}{\varpi}]$, it follows that $x$ is almost integral over $(V_W)^*$, as desired.
\end{proof}

\begin{lemma}
\label{almostintegral2}
Let $V$ be a valuation ring of rank one. If $V \to W$ is an almost integral extension and $W$ is an integral domain, then $V \to W$ is an integral extension.
\end{lemma}

\begin{proof}
Notice that $Q(V) \to Q(W)$ is an algebraic field extension. Let $U$ be
the integral closure of $V$ in $Q(W)$ and let
$\MSpec(U)$ denote the set of all maximal ideals. Then
for any maximal ideal $\fm \in \MSpec(U)$,
$U_{\mathfrak{m}}$ is a valuation ring of rank one by Proposition \ref{valuationext}. Now we can write
$U=\bigcap _{\fm \in \MSpec(U)} U_{\mathfrak{m}}$ by
\cite[Theorem 4.7]{Mat86}. Since $U_{\mathfrak{m}}$ is completely integrally
closed in $Q(U_{\mathfrak{m}})$, $U$ is completely integrally
closed in $Q(W)$. Since $V \to W$ is an almost integral extension and $W \subseteq Q(W)$, it follows that $W \subseteq U$. Therefore every element
of $W$ is integral over $V$. 
\end{proof}

\subsection{Normalized length over a valuation ring of rank one}
Let us briefly recall the \textit{normalized length} that is defined in the category of torsion modules over a valuation ring of rank one. So let $V$ be a valuation ring of rank one with a unique maximal ideal $\fm_V$, which is assumed to be either discrete or non-discrete. Denote by $\textbf{Mod}_{\{\fm_V\}}$ the category of $V$-modules with support at the maximal ideal $\{\fm_V\}$. For $M \in \textbf{Mod}_{\{\fm_V\}}$, one can define a well-behaved length
$$
\lambda_{\infty}(M) \in \mathbf{R}_{\ge 0} \cup \{\infty\}.
$$
If $M$ does not belong to $\textbf{Mod}_{\{\fm_V\}}$, then put $\lambda_\infty(M)=\infty$ for completeness of logic. Now we recall the definition and some properties. For more details, see \cite{GR07}, \cite{GR18}, \cite{Sch13} and \cite{Shi07}.

First of all, we start with finitely generated torsion $V$-modules. Let $M \in \textbf{Mod}_{\{\fm_V\}}$ be a finitely generated module. Then the $0$-th Fitting ideal $F_{0}(M)^a$ is well-defined (\cite[Lemma 2.3.20]{GR03}). And let $\Div(V^{a})$ be the topological group of non-zero fractional ideals of $V^{a}$ defined in \cite[6.1.16]{GR03} and let $\widehat{\Gamma_{V}}$ be the completion of the value group of $V$ with respect to some uniform structure defined in \cite[6.1.18]{GR03}. Then for every non-zero ideal $I \in \Div(V^{a})$, there exist some filtered ordered set $S$ and principal ideals $\{J_{i}~|~i \in S\}$, which converge to $I$. Let $\gamma_{i}$ be the value of a generator of $J_{i}$. Then $\{\gamma_{i}\}$ converge to some element $\hat{\gamma} \in \widehat{\Gamma_{V}}$. So assigning $I \mapsto |I| := \hat{\gamma}$, we obtain the homomorphism of topological groups $|\cdot| : \Div(V^a) \to \widehat{\Gamma_V}$. 

Moreover, this homomorphism $|\cdot|$ is an isomorphism by \cite[Lemma 6.1.19]{GR03}. Now remarking $F_{0}(M)^a \in \Div(V^a)$, we can define
$$
\lambda_{\infty}(M) := |F_{0}(M)^{a}| \in \log \widehat{\Gamma_{V}},
$$
where $\log\widehat{\Gamma_{V}}$ is the group regarding the multiplicative group $\Gamma_{V}$ as the additive group. Fix an embedding $\log \widehat{\Gamma_{V}} \hookrightarrow \mathbf{R}_{\geq 0}$ (Since $V$ is a rank one valuation ring, there exists an embedding $\Gamma_V \hookrightarrow \mathbf{R}_{\geq 0}$). Then we define the function $\lambda_\infty(M) \in \mathbf{R}_{\geq 0}$. 

Next let $M \in \textbf{Mod}_{\{\fm_V\}}$ be an arbitrary module. Since we have the inequality $\lambda_{\infty}(N) \leq \lambda_{\infty}(M)$ for a finitely generated submodule $N \subseteq M$, we define the function $\lambda_\infty(M) \in \mathbf{R}_{\geq 0}$ as follows:
$$
\lambda_{\infty}(M):=\sup\big\{\lambda_{\infty}(N)~\big|~N \subseteq M~\text{is a finitely generated $V$-module}\big\} \in \mathbf{R}_{\geq 0} \cup \{\infty\}.
$$

The normalized length has the following properties.

\begin{proposition}
\label{normalized1}
The following assertions hold.
\begin{enumerate}
\item
Let $0 \to L \to M \to N \to 0$ be a short exact sequence of $V$-modules in $\rm{\bf{Mod}}_{\{\fm_V\}}$. Then
$$
\lambda_{\infty}(M)=\lambda_{\infty}(L)+\lambda_{\infty}(N).
$$

\item
For the above short exact sequence and $a, b \in \fm_{V}$, 
$$
\lambda_{\infty}(abM) \leq \lambda_{\infty}(aL)+\lambda_{\infty}(bN).
$$

\item
If $\lambda_{\infty}(M)=0$, then $M$ is $\fm_V$-almost zero, and if $M$ is contained in a finitely presented $V$-module in $\rm{\bf{Mod}}_{\{\fm_V\}}$, then $M=0$.

\item
If a $V$-module $M$ is an ${\fm_{V}}$-almost finitely generated module in $\rm{\bf{Mod}}_{\{\fm_V\}}$, then $\lambda_\infty(bM)<\infty$ for any $b \in \fm_{V}$.
\end{enumerate}
\end{proposition}

\begin{proof}
$(1)$ is \cite[Proposition 14.5.12]{GR18}, $(2)$ is \cite[Lemma 14.5.80]{GR18}, $(3)$ is \cite[Proposition 14.5.12]{GR18} and \cite[Theorem 14.5.75]{GR18}, and $(4)$ is \cite[Lemma 14.5.83]{GR18}.
\end{proof}

\begin{remark}
Even if $\lambda_\infty(M)=0$, it may not be true that $M=0$. Such an example is provided by $M=V/\fm_V$.
\end{remark}

The most important part of normalized length is contained in the following theorem.

\begin{theorem}[Frobenius pull-back formula]
\label{Frobeniuspullback}
Suppose that $M$ is a $V/pV$-module. Then
$$
\lambda_\infty(M^{[F]})=\frac{1}{p} \cdot \lambda_\infty(M).
$$
Here, $M^{[F]}$ is regarded as a $V$-module through the restriction of scalars via the Frobenius endomorphism on $V/pV$.
\end{theorem}

\begin{proof}
This is \cite[Proposition 4.3.15]{GR07}.
\end{proof}

\begin{proposition}
\label{almostfinite}
Assume that $V$ is a valuation ring of rank one with field of fractions $K$ and a basic setup $(V,I)$. Let $W$ be the integral closure of $V$ in a finite separable extension of $K$. Then $W$ is an $I$-almost finitely presented $V$-module.
\end{proposition}

\begin{proof}
The proof uses only basic properties of trace mapping and Fitting ideals. For details, we refer the reader to \cite[Proposition 6.3.8]{GR03}.
\end{proof}

\section{The almost purity theorem in perfectoid valuation ring}
\label{perfectoidvaluation}
\subsection{Perfectoid valuation rings}

We introduce \textit{perfectoid valuation rings}.

\begin{definition}
\label{perfdval}
Let $p>0$ be a prime and let $V$ be a valuation ring such that $pV \ne V$. We say that $V$ is a \textit{perfectoid valuation ring} if the following hold:
\begin{enumerate}
\item
The valuation of $V$ is of rank one and not discrete.

\item
$V$ is complete in the valuation topology.

\item
The Frobenius endomorphism on $V/pV$ is surjective.
\end{enumerate}
\end{definition}

The definition includes the case $\mathbf{F}_p \subset V$, in which case $V$ is a perfect valuation ring. Remark that $(V,\fm_{V})$ is a basic setup. The following proposition follows from \cite[Lemma 3.2]{Sch12}
\begin{proposition}
\label{pthpower}
Let $V$ be a perfectoid valuation ring. Then the followings are hold.
\begin{enumerate}
\item
There exists an element $\varpi \in V$ such that $\varpi^{p}=pu$ for some $u \in V^{\times}$.
\item
For an element $\varpi \in V$ as in (1), $V$ is $\varpi$-adically complete and separated .
\end{enumerate}
\end{proposition}

Next we define the tilt for a ring.

\begin{definition}
Let $V$ be a ring and let $p>0$ be a prime number. Then
$$
V^{\flat}:=\vpl \{ \cdots \xrightarrow{F_V} V/pV \xrightarrow{F_V} V/pV \}
$$
is called the \textit{tilt} of $V$, where $F_{V}$ is the Frobenius endomorphism on $V/pV$.
\end{definition}


The tilting operation is well-behaved for perfectoid valuation rings. See \cite[Lemma 3.4]{Sch12} for the next proposition.

\begin{proposition}
\label{perf}
Let $V$ be a perfectoid valuation ring and let $\varpi \in V$ be an element as in Proposition \ref{pthpower}. The following assertions hold.

\begin{enumerate}
\item
There exists an isomorphism $\vpl_{x \mapsto x^{p}} V \cong V^{\flat}$ as multiplicative monoids. Moreover, for an element $\varpi \in V$ as in Proposition \ref{pthpower}, one can choose a compatible system of elements $\{ \varpi^{\frac{1}{p^{n}}}\}_{n \geq 0}$ in V such that $\varpi^{\flat}:=(\ldots,\overline{\varpi^{\frac{1}{p}}},\overline{\varpi},0) \in V^{\flat}$ is well-defined.

\item
$V^{\flat}$ is a perfectoid valuation ring and $\varpi^{\flat}$-adically complete and separated, where $\varpi^{\flat}$ is an element defined as in (1).

\item
We have an isomorphism
$$
V^{\flat} \cong \vpl \{ \cdots \xrightarrow{F_{V}} V/\varpi V \xrightarrow{F_{V}} V/\varpi V \}.
$$
\item
There is a short exact sequence
$$
0 \to V^{\flat} \xrightarrow{\times \varpi^{\flat}} V^{\flat} \xrightarrow{\pr_{V}} V/\varpi V \to 0,
$$
where $\pr_{V}$ is the projection onto the first factor. Moreover, this short exact sequence gives an isomorphism $V/\varpi V \cong V^{\flat}/{\varpi}^{\flat}V^{\flat}$.
\end{enumerate}
\end{proposition}

\subsection{Proof of the main theorem}
We recall that a ring map $f:A \to B$ is \textit{weakly \'etale} if both $A \to B$ and the diagonal map $\mu_{B}:B \otimes_{A} B \to B$ are flat (The latter is called \textit{weakly unramified}).

\begin{lemma}$($\cite[Theorem 3.5.13.]{GR03}$)$
Let $f:A \to B$ be a ring map of $\mathbf{F}_{p}$-algebras. If the map $f$ is weakly \'etale, then $F_{B/A} : A^{[F]} \otimes_{A} B \to B^{[F]}$ is an isomorphism. 
\end{lemma}

\begin{proof}
This is \cite[Theorem 3.5.13]{GR03}.
\end{proof}

\begin{lemma}
\label{surjective}
Let $f:A \to B$ be a ring homomorphism.
Then $f$ is surjective if and only if there exists a faithfully flat $A$-algebra $C$ such that $C \to C \otimes_{A} B$ is surjective.
\end{lemma}

\begin{proof}
If $f$ is surjective, we just take $C=A$. If there exists a faithfully flat $A$-algebra $C$ such that $C \to C \otimes_{A} B$ is surjective, $C \otimes_{A} \coker f=0$. Since $C$ is faithfully flat, we have $\coker f=0$.
\end{proof}

\begin{lemma}
\label{proj}
Let $(V,I)$ be a basic setup as in Proposition \ref{almostNAK} and let $M$ be an almost finitely generated $V$-module.  If $M$ is an almost projective module, then the following holds. Assume that $\phi : N \to M$ is a homomorphism of $V$-modules such that $\coker{\phi}$ is killed by $\varpi^{\alpha}$ for a fixed positive rational number $\alpha \in \mathbf{Q}_{>0}$. Then for any $k>0$, there exists $g : M \to N$ such that $\phi \circ g = \varpi^{\frac{1}{p^{k}}+2\alpha} \id_{M}$.
\end{lemma}

\begin{proof}
Take the exact sequence $N \to M \to \coker{\phi} \to 0$. Let $M'$ be the image of $\phi$. Then we obtain the short exact sequence $0 \to K \to N \to M' \to 0$, which induces the long exact sequence
$$
0 \to \Hom_{V}(M',K) \to \Hom_{V}(M' , N) \to \Hom_{V}(M' , M') \to \Ext^{1}_{V}(M' , K) \to \cdots.
$$
Next we take $0 \to M' \to M \to M/M' \to 0$ and it induces the long exact sequence
$$
\cdots  \to \Ext^{1}_{V}(M,K) \to \Ext^{1}_{V}(M',K) \to \Ext^{2}_{V}(M/M',K) \to \cdots.
$$
Then $\varpi^{\frac{1}{p^k}}\Ext^{1}_{V}(M,K) = 0$ for any $k>0$ by assumption and $\varpi^{\alpha} \Ext^{2}_{V}(M/M',K)=0$. So we obtain $\varpi^{\frac{1}{p^{k}}+\alpha} \Ext^{1}_{V}(M',K)=0$. Considering the first long exact sequence, we can find an element $f \in \Hom_{V}(M',N)$ such that $\phi \circ f = \varpi^{\frac{1}{p^{k}}+\alpha}\id_{M'}$. Finally, since $\varpi^{\alpha}M \subset M'$, we obtain $g=f \circ \varpi^{\alpha}$, as desired. 
\end{proof}

\begin{proposition}
\label{proj2}
Let $(V,I)$ be a basic setup as in Proposition \ref{almostNAK} and let $M$ be an almost finitely generated $V$-module such that $M/\varpi^{n}M$ is an almost projective $V/\varpi^{n}V$-module for any $n >0$. Suppose that $M$ and $V$ are $\varpi$-adically complete. Then $M$ is an almost projective $V$-module.
\end{proposition}

\begin{proof}
Since $M$ is almost finitely generated by assumption, we can find a $V$-module homomorphism $p_n:P_{n} \to M$ for any $n>0$ such that $P_{n}$ is a finite free module and $\varpi^{\frac{1}{p^{n}}} \coker{p_n}=0$. Remark that $P_n$ is $\varpi$-adically complete. Fix an integer $k>0$. Since $M/\varpi M$ is almost projective, there is a $V/\varpi V$-module map $s_{n_{1}}:M/\varpi M \to P_{n}/\varpi P_{n}$ such that
\begin{equation}
\label{almostsplit}
(p_n\text{ mod }\varpi) \circ s_{n_{1}}=\varpi^{\frac{2}{p^{n}}+\frac{1}{p^k}}\id_{M/\varpi M}
\end{equation}
by Lemma \ref{proj}. We will construct an almost splitting $s_{n_{2}}:M/\varpi^{2} M \to P_{n}/\varpi^{2} P_{n}$ out of $s_{n_{1}}$. Set $N_{2} := \pi^{-1}(\im(s_{n_1}))$, where $\pi : P_n \to P_n/\varpi P_{n}$. Then $N_2$ is a submodule of $P_n$ and the induced map $N_{2} \to M$ induces $t_1:N_{2}/\varpi N_{2} \to M/\varpi M$ and $t_2:N_{2}/\varpi^{2}N_{2} \to M/\varpi^{2}M$. By (\ref{almostsplit}), we obtain $\varpi^{\frac{2}{p^{n}}+\frac{1}{p^k}} \coker t_{1}=0$. Since $\coker t_{2}/\varpi \coker t_{2} \cong \coker t_{1}$, we obtain $\varpi^{\frac{2}{p^{n}}+\frac{1}{p^k}} \coker t_{2}=\varpi \coker t_{2}$. By using this identity repeatedly, we get
$$
\varpi^{\frac{2}{p^{n}}+\frac{1}{p^k}} \coker t_{2} = \varpi^{l-(l-1)\big(\frac{2}{p^{n}}+\frac{1}{p^k}\big)} \coker t_{2}.
$$
Since $l-(l-1)\big(\frac{2}{p^{n}}+\frac{1}{p^k}\big) \geq 2$ for $l \gg 0$, we have $\varpi^{\frac{2}{p^{n}}+\frac{1}{p^k}}\coker t_{2} = 0$. Since $M/\varpi^{2}M$ is an almost projective $V/\varpi^{2}V$-module, there is a $V/\varpi^{2}V$-module map $g: M/\varpi^{2}M \to N_{2}/\varpi^{2}N_{2}$ such that $t_2 \circ g=\varpi^{\frac{4}{p^{n}}+\frac{3}{p^{k}}}\id_{M/\varpi^{2}M}$ by Lemma \ref{proj}. Then the composition of $g$ with the natural map $N_{2}/\varpi^{2}N \to P_{n}/\varpi^{2}P_{n}$ will yield $s_{n_{2}}$. In a similar manner, we can construct $s_{n_{l+1}}$ from $s_{n_{l}}$ inductively. Finally, $P_{n}$ and $M$ are $\varpi$-adically complete, so we obtain $s_{n} =\vpl_{l>0} s_{n_{l}}$ such that $p_{n} \circ s_{n}=\varpi^{\frac{4}{p^{n}}+\frac{3}{p^{k}}} \id_{M}$. As $n, k$ are chosen arbitrarily, we conclude that $M$ is almost projective.
\end{proof}

\begin{theorem}
\label{AlmostPurityTheorem}
Let $V$ be a perfectoid valuation ring with field of fractions $K$, and let $W$ be the integral closure of $V$ in a finite \'etale extension $K \to L$.
\begin{enumerate}

\item
$W$ is also a perfectoid valuation ring. 

\item
$V \to W$ is faithfully flat and almost finite \'etale.
\end{enumerate}
\end{theorem}

\begin{proof}
First we prove the both assertions in the prime characteristic $p>0$ case. Remark that $W$ is a valuation ring of rank one and its valuation is induced by the valuation of $V$. In order to prove the completeness of $W$, it suffices to prove that $W$ is complete and separated in the $\varpi$-adic topology, where $\varpi \in V$ is an element that defines the adic topology of $V$. Since $\varpi$ is topologically nilpotent in $W$, $W$ is equipped with the $\varpi$-adic topology. Applying \cite[Proposition 4.1]{NS20}, we find that $W$ is also $\varpi$-adically complete and separated. Finally, since $K$ is perfect and $K \hookrightarrow L$ is a finite \'etale extension, it follows that $L$ is perfect and $W$ is thus perfect. The Frobenius endomorphism on $W/pW$ is surjective. So $(1)$ holds. Next, since $V \to W$ is flat and local, it is faithfully flat. So it suffices to show that $V \to W$ is almost unramified. Since $V$ and $W$ are microbial valuation rings, we have $K=V[\frac{1}{\varpi}]$ and $L=W[\frac{1}{\varpi}]$. There are several ways for proving the almost unramifiedness of $V \to W$, which we briefly mention below and indicate the sources.
\begin{enumerate}
\item[$\bullet$]
The almost purity theorem in prime characteristic. The proof of this fact uses only basic facts on the trace maps. The details are found in \cite[Theorem 3.5.28]{GR03}.

\item[$\bullet$]
Ramification theory and differentials. This proof seems to work for only valuation rings. The proof is found in \cite[Proposition 6.6.2]{GR03}.

\item[$\bullet$]
Galois theory of commutative rings. This idea was suggested in Andr\'e's paper (see \cite[Proposition 3.4.2]{An18}).
\end{enumerate}

So the theorem is proved in the equal characteristic $p>0$ case in an elementary way. Henceforth, we assume that $K$ has characteristic $0$. 

$(1)$: By the same reasoning as in the characteristic $p>0$ case, $W$ is a non-discrete valuation ring of rank one and complete in the valuation topology. So it suffices to prove that the Frobenius endomorphism on $W/pW$ is surjective. First, note that the Frobenius surjectivity on $V/pV$ implies that the residue field of $V$ is perfect. Noting that the strict henselization $V \to V^{\rm{sh}}$ is an ind-\'etale local extension, $V^{\rm{sh}}$ is a valuation ring of rank one and $F_{V^{\rm{sh}}}$ is surjective. Put $W':=V^{\rm{sh}} \otimes_{V} W$. Then $W'$ is a reduced and normal ring such that $W \to W'$ is ind-\'etale. In particular, the relative Frobenius map
\begin{equation}
\label{relativeFrob}
(W/pW)^{[F]} \otimes_{W/pW} W'/pW' \to (W'/pW')^{[F]}
\end{equation}
is an isomorphism. So $Q(W') \cong L_{1} \times \cdots \times L_{n}$ for some $n>0$ and $W' \cong W_{1} \times \cdots \times W_n$, where each $W_{i}$ is the integral closure of $V^{\rm{sh}}$ in $L_i$. Noticing the Frobenius surjectivity on $V^{\rm{sh}}/pV^{\rm{sh}}$, suppose that we have proved that the Frobenius endomorphism $F_{W_i}$ is surjective. Then $F_{W'}$ is surjective and in view of Lemma \ref{surjective} and $(\ref{relativeFrob})$, $F_{W}$ is surjective. So after replacing $V$ by the completion $\widehat{V^{\rm{sh}}}$, we may assume that the residue field of $V$ is algebraically closed. Let $\varpi$ be the element of $V$ provided by Proposition \ref{perf}. Set
$$
A:=V/\varpi V,~B:=W/\varpi W~\mbox{and}~C:=W/pW.
$$
Then the Frobenius endomorphism on $C$ induces an injection $B \hookrightarrow C^{[F]}$, and this injection induces the exact sequence of $A$-modules:
\begin{equation}
\label{exact1}
0 \to B \xrightarrow{\overline{F}_W} C^{[F]} \to N \to 0,
\end{equation}
where $N$ is the cokernel of $\overline{F}_W$. Our goal is to show that $N$ is almost zero.\footnote{One can also show that $N$ is almost zero by Proposition \ref{ValuationRing}. See \S5.} To this aim, let us fix an arbitrary element $b \in \fm_V$. Then $\overline{F}_W$ restricts to an injection $bB \to (b^pC)^{[F]}$, which gives another exact sequence:
\begin{equation}
\label{exact11}
0 \to bB \xrightarrow{\overline{F}_W} (b^pC)^{[F]} \to N_b \to 0
\end{equation}
with $N_b$ its cokernel. Combining $(\ref{exact1})$ and $(\ref{exact11})$, we obtain a commutative diagram:
$$
\begin{CD}
0 @>>> bB @>\overline{F}_W>> (b^pC)^{[F]} @>>> N_b @>>> 0 \\
@. @VVV @VVV @VVV \\
0 @>>> B @>\overline{F}_W>> C^{[F]} @>>> N @>>> 0 \\
\end{CD}
$$
whose vertical maps in the left and in the center are natural injections. By the snake lemme, it follows that the kernel and cokernel of the $A$-module map $N_b \to N$ are annihilated by $b$. Remark that we have $B \in \textbf{Mod}_{\{\fm_V\}}$ and $\lambda_\infty(b B)<\infty$ for any $b \in \fm_V$ by Proposition \ref{normalized1} and Proposition \ref{almostfinite}. Taking the length, we get
\begin{equation}
\label{exact2}
\lambda_\infty((b^pC)^{[F]})=\lambda_\infty(bB)+\lambda_\infty(N_b).
\end{equation}
It follow from Theorem \ref{Frobeniuspullback} that
\begin{equation}
\label{exact3}
\lambda_\infty((b^pC)^{[F]})=p^{-1} \cdot \lambda_\infty(b^pC).
\end{equation}
Since $C$ admits a filtration given by $\Fil^\bullet(W):=(\varpi^kC~|~0 \le k \le p)$, each of whose graded component of the associated graded module $\rm{gr}^\bullet(C)$ is isomorphic to $B$, we deduce from Proposition \ref{normalized1} that
\begin{equation}
\label{exact4}
\lambda_\infty(b^pC) \le p \cdot \lambda_\infty(bB).
\end{equation}
Combining $(\ref{exact3})$ and $(\ref{exact4})$, we obtain
$$
\lambda_\infty((b^pC)^{[F]}) \le \lambda_\infty(bB).
$$
Then this together with $(\ref{exact2})$ implies that $\lambda_\infty(N_b)=0$. That is to say, $\fm_VN_b=0$. On the other hand, both kernel and cokernel of $N_b \to N$ are annihilated by $b$ and $b \in \fm_V$ is arbitrary, we find that $\fm_V N=0$.

Since $V/\fm_V$ is algebraically closed and the unique maximal ideal of $W$ is $\fm_V W$ (this follows from the fact that the valuation of $V$ is non-discrete of rank one), it follows that $V/\fm_V \to W/\fm_V W$ is a trivial extension. In other words,
\begin{equation}
\label{exact5}
W=\fm_V W+V.
\end{equation}
Pick an arbitrary element $c \in W$. Then $(\ref{exact5})$ gives us an element $d \in V$ such that $c-d \in \fm_V W$. Let $\overline{c}-\overline{d}$ be the image of $c-d$ under $W \twoheadrightarrow W/pW$. Then since $\fm_V N=0$, it follows that $\overline{c}-\overline{d}$ is contained in the subring $\im(F_W) \subset W/pW$. However, as the Frobenius endomorphism on $V/pV$ is surjective, we have $\overline{d} \in V/pV \subset \im(F_W)$. Hence $\overline{c} \in \im(F_W)$ and $\im(F_W)=W/pW$, as desired.

$(2)$: By the same reasoning as in the positive characteristic case, it is true that $V \to W$ is faithfully flat. And by Proposition \ref{almostfinite}, $W$ is an almost finitely presented $V$-module. Hence it suffices to show that $V \to W$ is almost unramified. By the assertion $(1)$ and Proposition \ref{perf} (3), we have
\begin{equation}
\label{tiltingisom}
V/\varpi V \cong V^\flat/\varpi^\flat V^\flat~\mbox{and}~W/\varpi W \cong W^{\flat}/\varpi^{\flat}W^{\flat}.
\end{equation}
Then using $(\ref{tiltingisom})$, we see that $V^{\flat}/\varpi^{\flat}V^{\flat} \to W^{\flat}/\varpi^{\flat}W^{\flat}$ is almost finitely generated as a $V^\flat$-module homomorphism. So $V^{\flat} \to W^{\flat}$ is almost finitely generated by Proposition \ref{almostNAK}. We claim that $V^{\flat} \to W^{\flat}$ is an integral extension. To this aim, let $(V^\flat_{W^\flat})^*$ be the integral closure of $V^\flat$ in $W^\flat$. Then $(V^\flat_{W^\flat})^*$ is an integrally closed domain. Since $V^\flat$ is $\varpi^\flat$-adically complete, it is $\varpi^\flat$-adically Henselian and it follows that $(V^\flat_{W^\flat})^*$ is a local domain that is integral over $V^\flat$. Thus, $(V^\flat_{W^\flat})^*$ is a valuation ring of rank one and $(V^\flat_{W^\flat})^* \to W^\flat$ is almost finitely generated. It follows from Lemma \ref{almostintegral1} that $(V^\flat_{W^\flat})^* \to W^\flat$ is almost integral. Thus, it is an integral extension by Lemma \ref{almostintegral2}. We conclude that $V^{\flat} \to W^{\flat}$ is integral.

As $V^\flat$ is a perfect $\mathbf{F}_p$-algebra, $Q(V^{\flat}) \rightarrow Q(W^{\flat})$ is finite separable. Moreover, we remark that $W^{\flat}$ is the integral closure of $V^{\flat}$ in $Q(W^{\flat})$ since $W^{\flat}$ is a valuation ring, hence normal. Then the discussion in the beginning of the proof gives that $V^{\flat} \to W^{\flat}$ is almost finite \'etale. This fact, combined with $(\ref{tiltingisom})$, implies that $W/\varpi W\otimes_{V/\varpi V} W/\varpi W \to W/\varpi W$ is almost projective. On the other hand, the almost finite generatedness of $V \to W$ implies that $W/\varpi^{n} W \otimes_{V/\varpi^{n}V} W/\varpi^{n}W \to W/\varpi^{n}W$ is almost finitely presented for $n>0$. By \cite[Proposition 2.4.18 and Theorem 5.2.12(i)]{GR03}, $W/\varpi^{n} \otimes_{V/\varpi^{n}} W/\varpi^{n} \to W/\varpi^{n}$ is almost projective for any $n>0$.\footnote{One could alternatively apply \cite[Theorem 5.2.12(ii)]{GR03}. However, almost finite presentedness of $W \otimes_V W \to W$ allows us to use \cite[Theorem 5.2.12(i)]{GR03} whose proof is relatively simple.}  Applying \cite[Proposition 4.1]{NS20}, we find that $W\otimes_{V}W$ is $\varpi$-adically complete and separated, so $W$ is an almost projective $W \otimes_{V} W$-module in view of Proposition \ref{proj2}. It implies that $W$ is almost unramified over $V$, as desired.
\end{proof}

\section{Frobenius map on valuation rings}
\label{Valuationflat}
One of the main results of the paper \cite{RS16} states that if $V$ is a valuation domain such that $\mathbf{F}_p \subset V$, then the Frobenius endomorphism on $V$ is faithfully flat. We remark that any perfect $\mathbf{F}_p$-algebra has this property. In the mixed characteristic case, we have the following result.

\begin{proposition}
\label{ValuationRing}
Let $V$ be a valuation domain of mixed characteristic $p>0$. Assume that there is an element $\varpi \in V$ together with an element $u \in V^\times$ such that $\varpi^p=pu$. Then the Frobenius endomorphism on $V/pV$ induces a faithfully flat map $V/\varpi V \to (V/pV)^{[F]}$.
\end{proposition}

Before the proof, we apply this for Theorem \ref{AlmostPurityTheorem}(1), in particular the cokernel $N$ of $\overline{F_W}$ as in Theorem \ref{AlmostPurityTheorem}(1) is almost zero. Keeping notation as in Theorem \ref{AlmostPurityTheorem},  we know that $A \to B \to C^{[F]}$ is flat, so that we have $bA \otimes_A C^{[F]}=bC^{[F]}=(b^pC)^{[F]}$. By applying base change to $(\ref{exact1})$, we have an exact sequence $0 \to bB \to (b^pC)^{[F]} \to bA \otimes_A N \to 0$, and the length computation yields that $bA \otimes_A N$ and hence $N$ are almost zero.

\begin{proof}[Proof of Proposition \ref{ValuationRing}.]
This proof is inspired by \cite[Lemma 4.3.9]{GR07}. Let $I \subset V$ be an arbitrary nonzero finitely generated ideal. By \cite[Remark 6.1.12(i)]{GR03}, we may assume that $I=aV$, where $\varpi V \subset aV$.  By \cite[Theorem 7.8]{Mat86}, it suffices to show that $\Tor^{V/\varpi V}_1\big(V/I,(V/pV)^{[F]}\big)=0$. Then the short exact sequence $0 \to a(V/\varpi V) \to V/\varpi V \to V/aV \to 0$ induces an exact sequence:
\begin{eqnarray}
0 &\to& \Tor^{V/\varpi V}_1\big(V/aV,(V/pV)^{[F]}\big) \to a(V/\varpi V) \otimes_{V/\varpi V} (V/pV)^{[F]}  \xrightarrow{\phi} (V/pV)^{[F]} \nonumber \\
&\to& V/aV \otimes_{V/\varpi V} (V/pV)^{[F]} \to 0.  \nonumber
\end{eqnarray}
Therefore, it suffices to show that the map $\phi$ is injective. In other words, we show that the natural map
\begin{equation}
\label{isomo1}
a(V/\varpi V) \otimes_{V/\varpi V} (V/pV)^{[F]} \to a(V/pV)^{[F]} \cong (a^pV/pV)^{[F]}
\end{equation}
is an isomorphism. Since $\varpi V \subset aV$, there is an element $b \in V$ for which $\varpi=ab$. Then the natural surjection $V \xrightarrow{a} a(V/\varpi V)$ yields an isomorphism $V/bV \cong a(V/\varpi V)$. In view of the identities $\varpi^p=a^pb^p$ and $\varpi^p=pu$ with $u \in V^\times$, we get
\begin{eqnarray}
a(V/\varpi V) \otimes_{V/\varpi V} (V/pV)^{[F]} &\cong& V/bV \otimes_{V/\varpi V} (V/pV)^{[F]} \cong (V/pV)^{[F]}/b(V/pV)^{[F]} \nonumber \\
&\cong & (V/b^pV)^{[F]} \cong (a^pV/pV)^{[F]}, \nonumber
\end{eqnarray}
which is $(\ref{isomo1})$.
\end{proof}

This proposition is partly inspired by \cite[Theorem 3.2]{Lu19} which discusses the regularity of Noetherian rings via the Frobenius map. Indeed, Proposition \ref{ValuationRing} would follow if one knows that $V$ can be written as a colimit of regular subrings. This is expected to be true in general, as it is known as a consequence of Zariski's uniformization theorem. In \cite{Za40}, Zariski proved that any valuation ring containing $\mathbf{Q}$ is described as a colimit of smooth $\mathbf{Q}$-subalgebras. See also \cite{AD21}.

Suppose that $V=\varinjlim A_i$, where $A_i$ are regular. Without loss of generality, we may even assume that $\varpi \in A_i$. But then the Frobenius induces a faithfully flat map $A_i/\varpi A_i \to A_i/pA_i$ in view of \cite[Theorem 3.2]{Lu19}. The same holds for $\varinjlim A_i/\varpi A_i \to \varinjlim A_i/pA_i$. Therefore, $V/\varpi V \to V/pV$ is faithfully flat.

Finally we end this article with the following question.

\begin{question}
Is there a distinguished class of (non-Noetherian) rings $V$ for which the conclusion of Proposition \ref{ValuationRing} holds, other than the case that $V$ is a regular ring, a valuation ring, or a ring whose $p$-adic completion is an integral perfectoid ring?
\end{question}

\begin{acknowledgement}
The authors are grateful to the referee for reading the paper very carefully and providing many constructive comments. The second author was partially supported by JSPS Grant-in-Aid for Scientific Research(C) 18K03257.
\end{acknowledgement}

\end{document}